\documentclass[12pt,reqno,draft]{amsart} 
\usepackage{amssymb,amscd,url}

\begin{document}


\newif\ifdraft \drafttrue
\draftfalse

\newcommand{\DRAFTNUMBER}{3}
\newcommand{\DateTime}{June 2006}


\newtheorem{theorem}{Theorem}
\newtheorem{lemma}[theorem]{Lemma}
\newtheorem{conjecture}[theorem]{Conjecture}
\newtheorem{question}[theorem]{Question}
\newtheorem{proposition}[theorem]{Proposition}
\newtheorem{corollary}[theorem]{Corollary}
\newtheorem{claim}[theorem]{Claim}

\theoremstyle{definition}
\newtheorem*{definition}{Definition}
\newtheorem{example}[theorem]{Example}

\theoremstyle{remark}
\newtheorem{remark}[theorem]{Remark}
\newtheorem*{acknowledgement}{Acknowledgements}



\newenvironment{notation}[0]{%
  \begin{list}%
    {}%
    {\setlength{\itemindent}{0pt}
     \setlength{\labelwidth}{4\parindent}
     \setlength{\labelsep}{\parindent}
     \setlength{\leftmargin}{5\parindent}
     \setlength{\itemsep}{0pt}
     }%
   }%
  {\end{list}}

\newenvironment{parts}[0]{%
  \begin{list}{}%
    {\setlength{\itemindent}{0pt}
     \setlength{\labelwidth}{1.5\parindent}
     \setlength{\labelsep}{.5\parindent}
     \setlength{\leftmargin}{2\parindent}
     \setlength{\itemsep}{0pt}
     }%
   }%
  {\end{list}}
\newcommand{\Part}[1]{\item[\upshape#1]}

\renewcommand{\a}{\alpha}
\renewcommand{\b}{\beta}
\newcommand{\g}{\gamma}
\renewcommand{\d}{\delta}
\newcommand{\dhat}{{\hat\delta}}
\newcommand{\e}{\epsilon}
\newcommand{\f}{\varphi}
\newcommand{\bfphi}{{\boldsymbol{\f}}}
\renewcommand{\l}{\lambda}
\renewcommand{\k}{\kappa}
\newcommand{\lhat}{\hat\lambda}
\newcommand{\ltilde}{\tilde\lambda}
\newcommand{\m}{\mu}
\newcommand{\bfmu}{{\boldsymbol{\mu}}}
\renewcommand{\o}{\omega}
\renewcommand{\r}{\rho}
\newcommand{\rbar}{{\bar\rho}}
\newcommand{\s}{\sigma}
\newcommand{\sbar}{{\bar\sigma}}
\renewcommand{\t}{\tau}
\newcommand{\z}{\zeta}

\newcommand{\D}{\Delta}
\newcommand{\Dhat}{{\hat\Delta}}
\newcommand{\G}{\Gamma}
\newcommand{\F}{\Phi}

\newcommand{\ga}{{\mathfrak{a}}}
\newcommand{\gb}{{\mathfrak{b}}}
\newcommand{\gn}{{\mathfrak{n}}}
\newcommand{\gp}{{\mathfrak{p}}}
\newcommand{\gP}{{\mathfrak{P}}}
\newcommand{\gq}{{\mathfrak{q}}}

\newcommand{\Abar}{{\bar A}}
\newcommand{\Ebar}{{\bar E}}
\newcommand{\Kbar}{{\bar K}}
\newcommand{\Pbar}{{\bar P}}
\newcommand{\Sbar}{{\bar S}}
\newcommand{\Tbar}{{\bar T}}
\newcommand{\ybar}{{\bar y}}
\newcommand{\phibar}{{\bar\f}}

\newcommand{\Acal}{{\mathcal A}}
\newcommand{\Bcal}{{\mathcal B}}
\newcommand{\Ccal}{{\mathcal C}}
\newcommand{\Dcal}{{\mathcal D}}
\newcommand{\Ecal}{{\mathcal E}}
\newcommand{\Fcal}{{\mathcal F}}
\newcommand{\Gcal}{{\mathcal G}}
\newcommand{\Hcal}{{\mathcal H}}
\newcommand{\Ical}{{\mathcal I}}
\newcommand{\Jcal}{{\mathcal J}}
\newcommand{\Kcal}{{\mathcal K}}
\newcommand{\Lcal}{{\mathcal L}}
\newcommand{\Mcal}{{\mathcal M}}
\newcommand{\Ncal}{{\mathcal N}}
\newcommand{\Ocal}{{\mathcal O}}
\newcommand{\Pcal}{{\mathcal P}}
\newcommand{\Qcal}{{\mathcal Q}}
\newcommand{\Rcal}{{\mathcal R}}
\newcommand{\Scal}{{\mathcal S}}
\newcommand{\Tcal}{{\mathcal T}}
\newcommand{\Ucal}{{\mathcal U}}
\newcommand{\Vcal}{{\mathcal V}}
\newcommand{\Wcal}{{\mathcal W}}
\newcommand{\Xcal}{{\mathcal X}}
\newcommand{\Ycal}{{\mathcal Y}}
\newcommand{\Zcal}{{\mathcal Z}}

\renewcommand{\AA}{\mathbb{A}}
\newcommand{\BB}{\mathbb{B}}
\newcommand{\CC}{\mathbb{C}}
\newcommand{\FF}{\mathbb{F}}
\newcommand{\GG}{\mathbb{G}}
\newcommand{\PP}{\mathbb{P}}
\newcommand{\QQ}{\mathbb{Q}}
\newcommand{\RR}{\mathbb{R}}
\newcommand{\ZZ}{\mathbb{Z}}

\newcommand{\bfa}{{\mathbf a}}
\newcommand{\bfb}{{\mathbf b}}
\newcommand{\bfc}{{\mathbf c}}
\newcommand{\bfe}{{\mathbf e}}
\newcommand{\bff}{{\mathbf f}}
\newcommand{\bfg}{{\mathbf g}}
\newcommand{\bfp}{{\mathbf p}}
\newcommand{\bfr}{{\mathbf r}}
\newcommand{\bfs}{{\mathbf s}}
\newcommand{\bft}{{\mathbf t}}
\newcommand{\bfu}{{\mathbf u}}
\newcommand{\bfv}{{\mathbf v}}
\newcommand{\bfw}{{\mathbf w}}
\newcommand{\bfx}{{\mathbf x}}
\newcommand{\bfy}{{\mathbf y}}
\newcommand{\bfz}{{\mathbf z}}
\newcommand{\bfA}{{\mathbf A}}
\newcommand{\bfF}{{\mathbf F}}
\newcommand{\bfB}{{\mathbf B}}
\newcommand{\bfG}{{\mathbf G}}
\newcommand{\bfI}{{\mathbf I}}
\newcommand{\bfM}{{\mathbf M}}
\newcommand{\bfX}{{\mathbf X}}
\newcommand{\bfY}{{\mathbf Y}}
\newcommand{\bfZ}{{\mathbf Z}}
\newcommand{\bfzero}{{\boldsymbol{0}}}

\newcommand{\Aut}{\operatorname{Aut}}
\newcommand{\Disc}{\operatorname{Disc}}
\newcommand{\Div}{\operatorname{Div}}
\newcommand{\End}{\operatorname{End}}
\newcommand{\Fbar}{{\bar{F}}}
\newcommand{\Gal}{\operatorname{Gal}}
\newcommand{\GL}{\operatorname{GL}}
\newcommand{\Index}{\operatorname{Index}}
\newcommand{\Image}{\operatorname{Image}}
\newcommand{\hhat}{{\hat h}}
\newcommand{\Ker}{{\operatorname{ker}}}
\newcommand{\MOD}[1]{~(\textup{mod}~#1)}
\newcommand{\Norm}{{\operatorname{\mathsf{N}}}}
\newcommand{\notdivide}{\nmid}
\newcommand{\normalsubgroup}{\triangleleft}
\newcommand{\odd}{{\operatorname{odd}}}
\newcommand{\onto}{\twoheadrightarrow}
\newcommand{\ord}{\operatorname{ord}}
\newcommand{\PGL}{\operatorname{PGL}}
\newcommand{\Pic}{\operatorname{Pic}}
\newcommand{\Prob}{\operatorname{Prob}}
\newcommand{\Qbar}{{\bar{\QQ}}}
\newcommand{\rank}{\operatorname{rank}}
\newcommand{\Rat}{\operatorname{Rat}}
\newcommand{\Resultant}{\operatorname{Resultant}}
\renewcommand{\setminus}{\smallsetminus}
\newcommand{\Span}{\operatorname{Span}}
\newcommand{\tors}{{\textup{tors}}}
\newcommand{\Trace}{\operatorname{Trace}}
\newcommand{\UHP}{{\mathfrak{h}}}    

\newcommand{\longhookrightarrow}{\lhook\joinrel\longrightarrow}
\newcommand{\longonto}{\relbar\joinrel\twoheadrightarrow}

\newcommand{\OO}{{\mathcal{O}}}
\newcommand{\FS}{{\textup{FS}}}
\newcommand{\Spec}{\operatorname{Spec}}
\newcommand{\acherncl}{\operatorname{\widehat{c}}}
\newcommand{\adeg}{\operatorname{\widehat{deg}}}


\title[The arithmetic complexity of morphisms]{Canonical heights and
the\\ arithmetic complexity of morphisms\\ on projective space}
\author{Shu Kawaguchi and Joseph H. Silverman}
\email{kawaguch@math.kyoto-u.ac.jp, jhs@math.brown.edu}
\address{Department of Mathematics, Faculty of Science, 
        Kyoto University, Kyoto, 606-8502, Japan}
\address{Mathematics Department, Box 1917
        Brown University, Providence, RI 02912 USA}
\thanks{The first author thanks the Institut de Math\'ematiques de Jussieu
  and Vincent Maillot for their warm hospitality.}
\thanks{The second author's research supported by NSA grant
  H98230-04-1-0064}
\date{\DateTime\ifdraft, Draft \#\DRAFTNUMBER\fi}



\maketitle

\hbox to\hsize{\hfill
\emph{For John Tate on the occasion of his $80^{\textup{th}}$ birthday.}
\hfill}

\begin{abstract}
Let~$\f,\psi:\PP^N\to\PP^N$ be morphisms of degree~$d\ge2$
defined over~$\Qbar$.  We define the \emph{arithmetic distance}
between~$\f$ and~$\psi$ to be the supremum of the difference of their
canonical heights,
\[
  \dhat(\f,\psi) 
  = \sup_{P\in\PP^N(\Qbar)} \bigl|\hhat_\f(P)-\hhat_\psi(P)\bigr|.
\]
We prove comparison theorems relating~$\dhat(\f,\psi)$ to other height
functions and we show that for a fixed~$\psi$, the set of~$\f$
satisfying \text{$\dhat(\f,\psi)\le{B}$}
and \text{$\deg(\f)=d$} is a set of bounded height.
In particular, there are only finitely many such~$\f$ defined over any
given number field.
\end{abstract}


\section*{Prelude}
\label{section:prelude}

The theory of canonical heights on abelian varieties originated
with the work of N\'eron~\cite{MR0179173} and Tate (first
described in print by Manin~\cite{MR0173676}) in~1965.  Tate's simple and
elegant limit construction uses a Cauchy sequence telescoping sum
argument. N\'eron's construction, which is via more delicate local tools,
has proven to be fundamental for understanding the deeper
properties of the canonical height.
\par
Canonical heights appear prominently in the conjecture of Birch and
Swinnerton-Dyer, so early efforts to check the conjecture numerically
required the computation of~$\hhat(P)$ to at least a few decimal
places. In the mid-1970's, John Coates used Tate's limit
defi\-ni\-tion/con\-struc\-tion to compute~$\hhat(P)$ to three
decimal places and he mentioned the computation during a talk at the
Harvard Number Theory Seminar. That evening at a party in Coates'
honor, Tate pulled out a primitive Texas Instruments
programmable calculator, punched in a few values,
and in a fraction of a second recomputed Coates' value
of~$\hhat(P)$ to~8~decimal places! The
method was via a rapidly converging infinite series for N\'eron's
local canonical heights that Tate had described in an (unpublished)
letter to Serre.  Tate generously shared copies of his letter with other
mathematicians, including the second author of this paper (who was at
the time a mere graduate student), and Tate's numerically efficient
series for the computation of canonical heights appeared
in~\cite{MR866121} and, in generalized form, in~\cite{MR942161}.
\par
The importance of canonical heights in arithmetic geometry and related
fields has continued to grow, for example in arithmetic and
Arakelov intersection theory, special values of~$L$-functions,
cryptography, and dynamical systems. The present paper is devoted to
an application of the theory of canonical heights to study the
arithmetic properties of dynamical systems.

\section*{Introduction}
\label{section:introduction}
Let
\[
  h:\PP^N(\Qbar)\longrightarrow\RR
\]
be the standard (absolute logarithmic) Weil
height~\cite[\S{B.2}]{MR1745599} and let
\[
  \f:\PP^N\longrightarrow\PP^N
\]
be a morphism of degree~$d\ge2$ defined over~$\Qbar$.  Associated
to~$\f$ is a canonical  height function
\[
  \hhat_\f:\PP^N(\Qbar)\longrightarrow\RR
\]
defined via the Tate limit
\[
  \hhat_\f(P) = \lim_{n\to\infty} \frac{1}{d^n}h\bigl(\f^n(P)\bigr)
\]
and having the agreeable properties
\[
  \hhat_\f=h+O(1)
  \qquad\text{and}\qquad
  \hhat_\f\circ\f = d\hhat_\f.
\]
(See~\cite[\S{B.4}]{MR1745599} for details.) In an earlier
paper~\cite{KSEqualCanHt} the authors considered under what
circumstances two morphisms~$\f$ and~$\psi$ can have identical
canonical heights~$\hhat_\f=\hhat_\psi$. In this note we take
up the question of the extent to which the difference
\text{$\hhat_\f-\hhat_\psi$} is an intrinsic measure of the arithmetic
distance between the maps~$\f$ and~$\psi$. 
\par
More precisely, we define the \emph{arithmetic distance} between
two morphisms $\f,\psi:\PP^N\to\PP^N$ to be the quantity
\[
  \dhat(\f,\psi) 
    = \sup_{P\in\PP^N(\Qbar)} \left|\hhat_\f(P)-\hhat_\psi(P)\right|.
\]
Note that~$\dhat(\f,\psi)$ is finite, since $\hhat_\f=h+O(1)$ and
$\hhat_\psi=h+O(1)$. Further,~$\dhat(\f,\psi)=0$ if and only
if~$\hhat_\f=\hhat_\psi$, and an elementary application of the
triangle inequality (Lemma~\ref{lemma:triangular:inequality}) yields
\[
  \dhat(\f,\psi) \le \dhat(\f,\lambda)  + \dhat(\lambda,\psi).
\]
\par
The principal result in this note is a comparison theorem  showing
that the arithmetic distance between~$\f$ and~$\psi$ is related to the
naive height of~$\f$ and~$\psi$ as defined by the coefficients of
their defining polynomials. 
\par
More intrinsically, the set of rational
maps~$\PP^N\to\PP^N$ of degree~$d$, which we denote by~$\Rat_d^N$,
is naturally identified with a projective space via
\[
  \Rat_d^N\cong\PP^L,\qquad
  \f=[\f_0:\cdots:\f_N]\longmapsto
    [\text{coefficients of $\f_0,\ldots,\f_N$}].
\]
Defining~$h(\f)$, the \emph{Weil height of~$\f$}, to be the height of
the corresponding point in~$\PP^L$, we prove the following result. 

\begin{theorem}
\label{thm:mainthmintro}
Let~$\f,\psi:\PP^N\to\PP^N$ be morphisms of degree at least~$2$ defined
over~$\Qbar$. Then
\begin{equation}
  \label{eqn:dhatfpsill}
  \dhat(\f,\psi) - h(\psi) \ll h(\f) 
  \ll \dhat(\f,\psi) + h(\psi),
\end{equation}
where the implied constants depend only on~$N$ and the degrees of~$\f$
and~$\psi$. \textup(See Theorem~\ref{thm:comphd} for an explicit upper
bound.\textup)
\end{theorem}

An immediate corollary of Theorem~\ref{thm:mainthmintro} and
Northcott's theorem~\cite[B.2.3]{MR1745599} is the following
finiteness theorem.

\begin{corollary}
\label{cor:finitenessintro}
Fix a morphism~$\psi:\PP^N\to\PP^N$ of degree at least~$2$ defined
over~$\Qbar$ and an integer~$d\ge2$. Then for all~$B>0$ the set of
\emph{morphisms}~$\f\in\Rat_d^N(\Qbar)$ satisfying
\[
  \dhat(\f,\psi)\le B
\]
is a set of bounded height in~$\Rat_d^N(\Qbar)\cong\PP^L(\Qbar)$.
In particular, there are only finitely many such~$\f$ defined over
number fields of bounded degree over~$\QQ$.
\end{corollary}

The proof of Theorem~\ref{thm:mainthmintro} involves a number of steps.
It turns out to be more convenient to consider another sort of arithmetic
distance function defined by
\[
  \Dhat_\psi(\f) 
    = \sup_{P\in\PP^N(\Qbar)} 
       \left|\frac{1}{\deg(\f)}\hhat_\psi\bigl(\f(P)\bigr)-\hhat_\psi(P)\right|.
\]
An elementary argument relates~$\Dhat_\psi(\f)$ to~$\dhat(\f,\psi)$.
The triangle inequality for~$\dhat$ allows us to reduce to the case
that~$\psi$ is the powering map, so~$\hhat_\psi$ is the Weil
height~$h$, in which case we write simply~$\Dhat(\f)$.  Finally, and
this is the heart of the argument, we prove a theorem
comparing~$\Dhat(\f)$ to~$h(\f)$. For the lower bound we consider the
universal family of morphisms~$\PP^N\to\PP^N$ over~$\Rat_d^N$ and
apply general results of Call and Silverman~\cite{MR1255693}. For the
upper bound we use a matrix calculation to prove an explicit
inequality
\[
  h(\f) \le d\binom{N+d}{N}\Dhat(\f) + O_{N,d}(1).
\]

\section{Arithmetic Complexity}
\label{section:arithmeticomplexity}

The results proven in~\cite{KSEqualCanHt} show that morphisms with
identical canonical heights are very closely related to one
another. This suggests using the difference between canonical heights
as a way to measure the arithmetic distance between the morphisms,
which leads us to make the following definitions.

\begin{definition}
Let $\f,\psi:\PP^N\to\PP^N$ be morphisms defined over~$\Qbar$ of
degree at least~$2$.
We use the canonical height to define two \emph{arithmetic distance functions},
\begin{align*}
  \dhat(\f,\psi) 
    &= \sup_{P\in\PP^N(\Qbar)} \left|\hhat_\f(P)-\hhat_\psi(P)\right|, \\
  \Dhat_\psi(\f) 
    &= \sup_{P\in\PP^N(\Qbar)} 
       \left|\frac{1}{\deg(\f)}\hhat_\psi\bigl(\f(P)\bigr)-\hhat_\psi(P)\right|.
\end{align*}
In particular, $\dhat(\f,\psi)=0$ if and only if~$\hhat_\f=\hhat_\psi$.
\par
In the special case that~$\psi$ is the power map
\begin{equation}
  \label{eqn:powermap}
  [x_0,\ldots,x_N]\longmapsto[x_0^d,\ldots,x_N^d],
\end{equation}
so $\hhat_{\psi}$ is the usual Weil height~$h$, we write
simply~$\dhat(\f)$ and~$\Dhat(\f)$. We call~$\dhat(\f)$ the
\emph{arithmetic complexity of~$\f$}.
\end{definition}

\begin{remark}
The canonical heights~$\hhat_\f$ and~$\hhat_\psi$  satisfy
\[
  \hhat_\f = h+O(1)\qquad\text{and}\qquad\hhat_\psi=h+O(1)
\]
and the Weil height~$h$ satisfies $h\bigl(\f(P)\bigr)=\deg(\f){h}(P)+O(1)$,
so the su\-pre\-ma used to define~$\dhat(\f,\psi)$ and~$\Dhat_\psi(\f)$
are finite.
\end{remark}

\begin{remark}
The canonical height, and \emph{a fortiori} the arithmetic
distance~$\dhat(\f,\psi)$, are only defined for maps of degree at
least~$2$. However, we observe that~$\Dhat_\psi(\f)$ is well-defined
also for~$\deg(\f)=1$.
\end{remark}

We begin with an elementary triangle inequality for~$\dhat$.

\begin{lemma}
\label{lemma:triangular:inequality}
Let $\f,\psi,\nu:\PP^N\to\PP^N$ be morphisms defined over~$\Qbar$ of
degree at least~$2$. Then
\[
  \dhat(\f,\psi)\le\dhat(\f,\nu)+\dhat(\nu,\psi).
\]
\end{lemma}
\begin{proof}
This is immediate by taking suprema of
\[
  \bigl| \hat{h}_\varphi(P) - \hat{h}_\psi(P) \bigr| 
  \leq \bigl| \hat{h}_\varphi(P) - \hhat_\nu(P) \bigr| + 
  \bigl| \hhat_\nu(P) - \hat{h}_\psi(P)  \bigr|.
\]
\end{proof}

We next prove a comparison theorem for~$\dhat$ and~$\Dhat$. 

\begin{proposition}
\label{proposition:DdD}
Let $\f,\psi:\PP^N\to\PP^N$ be morphisms defined over~$\Qbar$ of
degree at least~$2$ and let~$d_\f=\deg(\f)$. Then
\[
  \frac{d_\f}{d_\f+1}\Dhat_\psi(\f) \le \dhat(\f,\psi)
  \le \frac{d_\f}{d_\f-1}\Dhat_\psi(\f).
\]
\end{proposition}
\begin{proof}
We compute directly using the definitions of~$\dhat$ and~$\Dhat$, the
triangle inequality, and basic properties of the canonical height.
\begin{align*}
  \Dhat_\psi&(\f)
  = \sup_{P\in\PP^N(\Qbar)} 
       \left|\frac{1}{d_\f}\hhat_\psi\bigl(\f(P)\bigr)-\hhat_\psi(P)\right| \\
  &\le \sup_{P\in\PP^N(\Qbar)} 
       \frac{1}{d_\f}
       \left|\hhat_\psi\bigl(\f(P)\bigr)-\hhat_\f\bigl(\f(P)\bigr)\right| 
   +  \sup_{P\in\PP^N(\Qbar)} 
       \left|\hhat_\f(P)-\hhat_\psi(P)\right| \\
  &=   \frac{1}{d_\f}\dhat(\f,\psi) + \dhat(\f,\psi).
\end{align*}
This gives one inequality. The other is proven similarly. Thus
\begin{align*}
  \Dhat_\psi&(\f)
  = \sup_{P\in\PP^N(\Qbar)} 
       \left|\frac{1}{d_\f}\hhat_\psi\bigl(\f(P)\bigr)-\hhat_\psi(P)\right| \\
  &\ge \sup_{P\in\PP^N(\Qbar)} 
       \left|\hhat_\psi(P)-\hhat_\f(P)\right|
    - \sup_{P\in\PP^N(\Qbar)} 
       \frac{1}{d_\f}
       \left|\hhat_\f\bigl(\f(P)\bigr)-\hhat_\psi\bigl(\f(P)\bigr)\right| \\
  &=   \dhat(\f,\psi) - \frac{1}{d_\f}\dhat(\f,\psi).
\end{align*}
This completes the proof of the proposition.
\end{proof}

\begin{remark}
Let $\bfphi = (\f_1, \f_2, \ldots)$ be a sequence of 
morphisms $\f_i: \PP^N \to \PP^N$ of degree 
$\deg(\f_i) \geq 2$. The {\em arithmetic complexity} of the 
sequence $\bfphi$ is the quantity 
\[
  \dhat(\bfphi) = \sup_{i \geq 1} \dhat(\f_i),
\]
and we say that the sequence $\bfphi$ is (\emph{arithmetically})
\emph{bounded} if $\dhat(\bfphi)$ is finite. It is shown in
\cite{arxiv0510633} that there is a canonical height
function~$\hhat_\bfphi$ naturally associated to every arithmetically
bounded sequence. A consequence of Theorem~\ref{thm:finiteness} proven below
is that if~$\bfphi$ is  arithmetically bounded
and contains infinitely many distinct maps, 
then either the degrees~$\deg(\f_i)$ of the maps or the
degrees of the fields of definition~$\QQ(\f_i)$ must go to infinity.
\end{remark}

\section{A comparison theorem for $h(\f)$ and $\Dhat(\f)$}
\label{section:comparehd}

The arithmetic complexity~$\dhat(\f)$ of a rational map is an
intrinsic measure of the extent to which~$\f$ differs arithmetically
from the elementary power map~\eqref{eqn:powermap}. 
A more naive way to measure the arithmetic complexity of~$\f$ is
to take the height of its coefficients. In this section we relate
these two notions. This will be used in the next section to 
show that, in a suitable sense, there are only finitely many rational
maps of bounded complexity. 
\par
We write~$\Rat_d^N$ for the set of rational maps~$\f:\PP^N \to \PP^N$
of degree~$d$. This set is naturally identified with a projective
space~$\PP^L$ by writing~$\f=[\f_0:\cdots:\f_N]$ and using the
coefficients of the homogeneous polynomials~$\f_0,\ldots,\f_N$ as
homogeneous coordinates in~$\PP^L$. The subset of~$\Rat_d^N$
consisting of morphisms is an affine subset of~$\PP^L$. (In fact, it
is the complement of a hypersurface.)

\begin{definition}
Let~$\f:\PP^N\to\PP^N$ be a rational map of degree~$d$ defined over~$\Qbar$.
We define the \emph{Weil height of~$\f$} to be the height of the corresponding
point in~$\Rat_d^N(\Qbar)\cong\PP^L(\Qbar)$. We denote this height
by~$h(\f)$.  Similarly, the \emph{height}~$h(F)$ of a nonzero
homogeneous polynomial~$F\in\Qbar[x_0, \ldots, x_N]$ is the height of
the point in projective space defined by its coordinates.
\end{definition}

\begin{theorem}
\label{thm:h:c}
Let~$N\ge1$ and~$d\ge1$ be given.  There are
constants $c_1,c_2,c_3>0$, depending only on~$N$ and~$d$, so that for
all morphisms $\f: \PP^N \to \PP^N$ of degree $d\geq 1$ defined over~$\Qbar$,
\begin{equation}
  \label{equation:hfledcf}
  c_1\Dhat(\f) - c_2 \leq h(\f) \leq d\binom{N+d}{N}  \Dhat(\f) + c_3.
\end{equation}
\end{theorem}

\begin{remark}
We give an example with $N=1$ that illustrates
the upper bound in the theorem.  Let~$\f_A(x)=x^d+Ax^{d-1}$
with~$A\in\ZZ$, $A\ne0$. Then one easily checks that for
all~$\a\in\Qbar$,
\[
  \left|\frac{1}{d}h\bigl(\f_A(\a))-h(\a)\right|
  \le \frac{1}{d}\log(1+|A|) = \frac{1}{d}h(\f_A)+O(1/|A|).
\]
Taking the supremum over~$\a$ yields
$d\Dhat(\f_A)\le h(\f_A)+O(1/|A|)$, and hence
\[
  \limsup_{|A|\to\infty} \frac{h(\f_A)}{\Dhat(\f_A)}\ge d.
\]
This may be compared with the upper bound of~$d^2+d$
provided by the theorem.  
\end{remark}

\begin{remark}
In general, we consider the limit
\[
  \a(N,d) = \limsup_{\substack{\f\in\Rat_{N,d}(\Qbar)\\h(\f)\to\infty\\}}
  \frac{h(\f)}{\Dhat(\f)} 
  \le d\binom{N+d}{N},
\]
where the upper bound is provided by Theorem~\ref{thm:h:c}. It would
be interesting to improve this upper bound and/or to obtain nontrivial
lower bounds for~$\a(N,d)$.
\end{remark}

\begin{proof}[Proof of the  Lower Bound in Theorem~\ref{thm:h:c}]
We prove the lower bound in~\eqref{equation:hfledcf} by showing
that it is a special case of~\cite[Theorem~3.1]{MR1255693}. In the
notation of~\cite{MR1255693}, we take~$T^0$ to be the set of
morphisms~$\PP^N\to\PP^N$ of degree~$d$. Thus~$T^0$ is naturally an
open subset of~$\Rat_d^N\cong\PP^L$ and we set~$T=\PP^L$. Then we
let~$\Vcal=\PP^N\times T$, we let~$\Vcal\to{T}$ be projection onto the
second factor, and we let~$\f:\Vcal\dashrightarrow\Vcal$ be the rational
map whose restriction to the generic fiber is the generic degree~$d$
morphism from~$\PP^N$ to itself. We further let~$\eta$ be a divisor
class in~$\Pic(\Vcal)$ whose restriction to the generic fiber is a
hyperplane section. Then~\cite[Theorem~3.1]{MR1255693} says that there
are (positive) constants~$c_1,c_2$ depending only on the family,
i.e., depending only on~$N$ and~$d$, so that
\begin{multline}
  \label{eqn:CSest}
  \bigl| \hhat_{\Vcal_t,\eta_t,\f_t}(x)-h_{\Vcal,\eta}(x)\bigr|
  \le c_1 h_T(t) + c_2 \\
  \qquad\text{for all $t\in T^0(\Qbar)$ and all $x\in\Vcal_t(\Qbar)$.}
\end{multline}
\par
Note that for each choice of~$t\in T^0(\Qbar)$, we get a degree~$d$
morphism $\f_t:\PP^N\to\PP^N$, and that $\hhat_{\Vcal_t,\eta_t,\f_t}$
is then our height function~$\hhat_{\f_t}$.  Further,~$h_{\Vcal,\eta}$
restricted to any particular fiber~$\Vcal_t=\PP^N$ is a Weil height
function on~$\PP^N$, and~$h_T(t)$ is simply the height~$h(\f_t)$ of
the morphism~$\f_t$. Thus~\eqref{eqn:CSest} becomes
\[
  \bigl| \hhat_{\f_t}(x) - h(x) \bigr| \le c_1 h(\f_t) + c_2
  \qquad\text{for all $x\in\PP^N(\Qbar)$.}
\]
Taking the supremum over~$x\in\PP^N(\Qbar)$ yields
\[
  \dhat(\f_t) \le c_1 h(\f_t) + c_2
\]
and then Proposition~\ref{proposition:DdD} gives
\[
  \Dhat(\f_t) \le c'_1 h(\f_t) + c'_2
\]
with $c_i'=(1+d^{-1})c_i$.
This inequality holds for all~$t\in T^0(\Qbar)$ with constants~$c'_1$
and~$c'_2$ that are independent of~$t$. By construction, as~$t$ varies
over~$T^0(\Qbar)$, the map~$\f_t$ varies over all degree~$d$
morphisms~$\PP^N\to\PP^N$.  This concludes the proof of the lower
bound in~\eqref{equation:hfledcf}.
\end{proof}

The idea underlying the proof of the upper bound in
Theorem~\ref{thm:h:c} is that a rational map~$\f:\PP^N\to\PP^N$ is
uniquely determined by its values at a sufficient number of generic
points~$P_1,\ldots,P_K$.  More precisely, the coefficients of the
polynomials defining~$\f$ are themselves polynomial functions of the
coordinates of~$\f(P_1),\ldots,\f(P_K)$. In order to obtain an
explicit upper bound in Theorem~\ref{thm:h:c}, we determine exactly
the degrees of these polynomial functions, which will enable
us to prove the following key estimate.

\begin{proposition}
\label{prop:hfledK}
Fix integers~$N\ge1$ and~$d\ge1$, and let $K=\binom{N+d}{d}$. There a
constant~$C_{N,d}$ and a Zariski closed set $Z_{N,d}\subset(\PP^N)^K$ so
that for all rational maps~$\f:\PP^N\to\PP^N$ of degree~$d$
defined over~$\Qbar$,
\begin{multline*}
  h(\f) 
  \le d(K-1)\sum_{j=1}^K h(P_j) + \sum_{j=1}^K h\bigl(\f(P_j)\bigr)
     + C_{N,d}\\
  \quad\text{for all $(P_1,\ldots,P_K)\in(\PP^N(\Qbar))^K\setminus Z_{N,d}$.}
\end{multline*}
\end{proposition}

\begin{proof}[Proof of Proposition~\textup{\ref{prop:hfledK}}]
We start by setting some notation.
\begin{notation}
\item[$\Ical_{N,d}$]
The set of $(N+1)$-tuples of nonnegative integers $(i_0,\ldots,i_N)$ 
satisfying $i_0+\cdots+i_N=d$.
\item[$\bfX^I$]
${}=M_I(\bfX)=X_0^{i_0}X_1^{i_1}\cdots{X}_N^{i_N}$, 
the monomial corresponding to the $(N+1)$-tuple $I=(i_0,\ldots,i_N)$.
\item[$K$]
${}=\binom{N+d}{N}=\#\Ical_{N,d}$, the number of monomials of degree~$d$
in~$N+1$~variables.
\end{notation}

\begin{lemma}
\label{lemma:nonzerodet}
Let~$\bfX^{(1)},\bfX^{(2)},\ldots,\bfX^{(K)}$ be $(N+1)$-tuples
whose~$(N+1)K$ coordinates are algebraically independent variables. Then
the matrix
\begin{equation}
  \label{eqn:MImatrix}
  \left( M_I(\bfX^{(j)}) \right)_{\substack{I\in\Ical_{N,d}\\ 1\le j\le K\\}}
\end{equation}
whose rows are the degree~$d$ monomials in the coordinates of
the~$\bfX^{(j)}$ has nonzero determinant. It is multihomogeneous of
degree~$d$ in each of~$\bfX^{(1)},\ldots,\bfX^{(K)}$.
\end{lemma}
\begin{proof}
The determinant is a sum of terms of the form
\[
  \pm M_{I_1}(\bfX^{(1)})\cdot M_{I_2}(\bfX^{(2)})\cdots 
  M_{I_K}(\bfX^{(K)})
  \quad\text{with ${I_1,\ldots,I_K}\in\Ical_{N,d}$.}
\]
Each of these terms is a distinct monomial in the polynomial ring
{\small
\[
  \ZZ\bigl[X^{(1)}_0,X^{(1)}_1,\dots,X^{(1)}_N,
    X^{(2)}_0,X^{(2)}_1,\dots,X^{(2)}_N,\dots,
    X^{(K)}_0,X^{(K)}_1,\dots,X^{(K)}_N
  \bigr].
\]
}%
Hence there can be no cancellation, so the determinant is nonzero (and
in fact is a sum/difference of~$K$~distinct monomials).  Finally, the
multihomogeneity is obvious, since each $M_I(\bfX^{(j)})$ is
homogeneous of degree~$d$ in the coefficients of~$\bfX^{(j)}$.
\end{proof}

Resuming the proof of Proposition~\ref{prop:hfledK}, we let
\[
  \f=[\f_0:\cdots:\f_N] : \PP^N\longrightarrow\PP^N
\]
be a rational map of degree~$d$, so each~$\f_i(\bfX)$ is a homogeneous
polynomial of degree~$d$. We write
\[
  \f_i(\bfX) = \sum_{I\in\Ical_{N,d}} a_{iI}M_I(\bfX),
\]
so the map~$\f$ corresponds to the point
\[
  [a_{iI}]_{\substack{I\in\Ical_{N,d}\\0\le i\le N\\}}
  \in \PP^{(N+1)K-1} \cong \Rat_{N,d}.
\]
\par
Let~$\bfX^{(1)},\ldots,\bfX^{(K)}$ be independent~$(N+1)$-tuples as
in Lemma~\ref{lemma:nonzerodet}
and consider the system of equations
\[
  \sum_{I\in\Ical_{N,d}} a_{iI}M_I(\bfX^{(j)}) = \f_i(\bfX^{(j)}),
  \qquad 1\le j\le K.
\]
We treat the~$\bfX^{(j)}$ as fixed quantities and solve for
the~$a_{iI}$~coefficients. To make this precise, let
$A=\left(M_I(\bfX^{(j)})\right)$ be the matrix~\eqref{eqn:MImatrix} 
defined in
Lemma~\ref{lemma:nonzerodet}, let~$B=A^{\text{adj}}$ be the adjoint
matrix, and let $D=\det(A)$. 
Then we obtain
\[
  D a_{iI} = \sum_{1\le j\le K} B_{jI} \f_i(\bfX^{(j)}),
  \qquad I\in\Ical_{N,d}.
\]
\par 
The coordinates of the~$j^{\text{th}}$~row of the matrix~$A$ are the
degree~$d$~monomials in the coordinates of~$\bfX^{(j)}$.  The
coordinates of the adjoint matrix \text{$B=A^{\text{adj}}$} are
sums/differences of terms, each of which is a product of~$K-1$ entries
from~$A$. More precisely, the entry~$B_{jI}$ is a sum/difference of
monomials, each of which is multihomogenous of degree~$d$ in the~$K-1$
variables
\[
  \bfX^{(1)},\bfX^{(2)},\ldots,\bfX^{(j-1)},\bfX^{(j+1)},\ldots,\bfX^{(K)}.
\]
For convenience we write~$B_{jI}(\bfX^{(1)},\ldots,\bfX^{(K)})$, although
in fact~$B_{jI}$ does not depend on~$\bfX^{(j)}$.
\par
We now define a rational map
\begin{gather*}
  U : (\PP^N\times\PP^N)^K \longrightarrow \PP^{(N+1)K-1} \\
  U(P_1,Q_1,\ldots,P_K,Q_K) 
    = 
     \biggl[\sum_{1\le j\le K} B_{jI}(P_1,\ldots,P_K) X_i(Q_j)
     \biggr]_{\substack{I\in\Ical_{N,d}\\0\le i\le N\\}}
\end{gather*}
where~$X_i(Q)$ denotes the~$i^{\text{th}}$~coordinate of~$Q$.
Notice that~$U$ is multihomogeneous of degree~$d(K-1)$ in
the variables~$P_1,\ldots,P_K$ and it is 
multihomogeneous of degree~$1$ in the variables~$Q_1,\ldots,Q_K$.
\par
From the way that we have set up these equations, we have for
all~$P_1,\ldots,P_K\in\PP^N$ 
\[
  U\bigl(P_1,\f(P_1),P_2,\f(P_2),\ldots,P_K,\f(P_K)\bigr) = \bigl[
  D(P_1,\ldots,P_K)a_{iI} \bigr]_{\substack{I\in\Ical_{N,d}\\0\le i\le N\\}}.
\]
Hence if~$P_1,\ldots,P_K\in\PP^N$ satisfy~$D(P_1,\ldots,P_N)\ne0$, then
\begin{align*}
  U\bigl(P_1,\f(P_1),P_2,\f(P_2),\ldots,&P_K,\f(P_K)\bigr)  \\
  &= [a_{iI}]= [\f] \in \Rat_{d,N} = \PP^{(N+1)K-1}.
\end{align*}
Taking heights and using the multihomogeneity of~$U$, we obtain
\begin{align*}
  h(\f) 
  &= h\bigl(U\bigl(P_1,\f(P_1)\ldots,P_K,\f(P_K)\bigr)\bigr)\\
  &\le d(K-1)\sum_{j=1}^K h(P_j) + \sum_{j=1}^K h\bigl(\f(P_j)\bigr)
    + O_{N,d}(1).
\end{align*}
Note that the inequality in this direction is a simple consequence of
the triangle inequality, we do not need~$U$ to be a
morphism. (Indeed,~$U$ is not a morphism.) And it would not be hard to
obtain an explicit value for the~$O_{N,d}(1)$ constant, although we
shall not do so.  
\end{proof}

In order to complete the proof of the upper bound in
Theorem~\ref{thm:h:c} we exploit the fact that the points of
height zero are Zariski dense, as described in the following lemma.

\begin{lemma}
\label{lemma:ZD}
Let~$V_1,\ldots,V_N$ be projective varieties, and for each~$i$,
let~$T_i\subset V_i$ be a Zariski dense set of points. Then the
product~$T=T_1\times\cdots\times{T}_N$ is Zariski dense
in~$V=V_1\times\cdots\times{V}_N$.
\par
In particular, 
\[
  \bigl\{ P \in \PP^N(\Qbar) : h(P) = 0 \bigr\}
\]
is Zariski dense in~$\PP^N$.
\end{lemma}
\begin{proof}
The proof is by induction on~$N$. For~$N=1$ the assertion is
clear. Assume it is true for~$N-1$.  Let~$f$ be a rational function
on~$V$ that vanishes on~$T$ and let~$t_1\in T_1$. (We may assume that
the support of the polar divisor of~$f$ does not contain the
set~$\{t_1\}\times V_2\times\cdots\times V_N$.)  Then by assumption,
the rational function
\[
  g_{t_1}(x_2,\ldots,x_N) = 
  f(t_1,x_2,\ldots,x_N)\quad\text{on $V_2\times\cdots\times V_N$}
\]
vanishes on $T_2\times\cdots\times T_N$. By induction, the
set~$T_2\times\cdots\times T_N$ is Zariski dense
in~$V_2\times\cdots\times V_N$, so we conclude that~$g_{t_1}$ is
identically~$0$ on~$V_2\times\cdots\times V_N$. Hence for any choice of
points~$(y_2,\ldots,y_N)\in V_2\times\cdots\times V_N$, the rational function
\[
  f(x_1,y_2,\ldots,y_N)\quad\text{on $V_1$}
\]
vanishes on~$T_1$. Since~$T_1$ is Zariski dense in~$V_1$, it follows
that if vanishes for all~$x_1\in V_1$. This proves that~$f$ is identically~$0$
on~$V$. Hence~$T$ is Zariski dense in~$V$.
\par
For the second statement, we observe that~$h(P)=0$ for those points
all of whose coordinates are either~$0$ or roots of unity.
Let~$\bfmu_\infty\subset\Qbar^*\subset\PP^1(\Qbar)$ denote the set of
all roots of unity.  Then~$\bfmu_\infty$ is infinite, so it is dense
in~$\PP^1$, and hence~$\bfmu_\infty^N$ is Zariski dense in~$(\PP^1)^N$.
But~$(\PP^1)^N$ is birational to~$\PP^N$, so
\[
  \bigl\{ [1:\z_1:\cdots:\z_N] : \z_1,\ldots,\z_N\in\bfmu_\infty \bigr\}
\]
is Zariski dense in~$\PP^N$.
\end{proof}

\begin{proof}[Proof of the Upper Bound in Theorem~\ref{thm:h:c}]
Let~$C_{N,d}$,~$Z_{N,d}$, and~$K$ be as in the statement of 
Proposition~\ref{prop:hfledK}. From second part
of Lemma~\ref{lemma:ZD}, we
know that the points of Weil height~$0$ are dense in~$\PP^N(\Qbar)$,
and then the first part of Lemma~\ref{lemma:ZD} tells us that
\begin{equation}
  \label{eqn:P1PK}
  \bigl\{ (P_1,\ldots,P_K) \in \PP^N(\Qbar)^K : h(P_1)=\cdots=h(P_K)=0 \bigr\}
\end{equation}
is Zariski dense in~$(\PP^N)^K$. In particular, we can find a
$K$-tuple of points~$(P_1,\ldots,P_K)$ in the set~\eqref{eqn:P1PK}
that is not in the Zariski closed set~$Z_{N,d}$.  Then
Proposition~\ref{prop:hfledK} gives the estimate
\[
  h(\f) \le \sum_{j=1}^K h\bigl(\f(P_j)\bigr) + C_{N,d}.
\]
Using the fact that every~$h(P_j)=0$, we rewrite this as
\begin{align*}
  h(\f) 
  &\le dK\max_{1\le j\le K} 
     \left(\frac{1}{d}h\bigl(\f(P_j)\bigr) - h(P_j)\right) + C_{N,d} \\
  &\le dK\sup_{P\in\PP^N(\Qbar)}
     \left|\frac{1}{d}h\bigl(\f(P)\bigr) - h(P)\right|  + C_{N,d}\\
  &= dK\Dhat(\f) + C_{N,d}.
\end{align*}
This completes the proof of the upper bound in Theorem~\ref{thm:h:c}.
\end{proof}

\begin{remark}
Our elementary proof of Proposition~\ref{prop:hfledK} is via a direct
matrix calculation.  Zhang~\cite[Theorem~5.2]{MR1254133} has proven
that if~$V\subset\PP^N$ is a variety defined over~$\Qbar$, then
\[
  h(V) \le \sup_{Z\subsetneq V} \inf_{P\in (V\setminus Z)(\Qbar)} h(P),
\]
where the supremum is over Zariski closed subsets of~$V$ and where the
height~$h(V)$ of the variety~$V$ is defined using arithmetic
intersection theory and the Fubini-Study metric on~$\PP^N$.
(See~\cite{MR1627110,MR1260106,MR1254133} for further details.)
Applying Zhang's inequality to various projections, we can prove a
version of Theorem~\ref{thm:h:c} of the form
\[
  h(\phi_i) \le Nd \Dhat(\phi) + O({N,d}(1)
  \qquad\text{for $i=0,1,\ldots,N$.}
\]
This is somewhat weaker than the upper bound~\eqref{equation:hfledcf}
in Theorem~\ref{thm:h:c}, since it involves the individual coordinate
functions of~$\f$, but the constant is better. It would be interesting
to try to use Zhang's inequality directly to prove
Proposition~\ref{prop:hfledK} and Theorem~\ref{thm:h:c}. One
possibility might be to apply Zhang's result to the graph
\[
  V=\bigl\{(P,\f(P)):P\in\PP^N\bigr\} \subset\PP^N\times\PP^N,
\]
thereby obtaining an estimate that simultaneously involves all of the
coordinate functions of~$\f$, but we will not pursue this idea further
in this paper.
\end{remark}

\section{Finiteness properties}

We combine the various comparison results to prove
Theorem~\ref{thm:mainthmintro}, which we restate here for the
convenience of the reader.

\begin{theorem}
\label{thm:comphd}
Let~$\f,\psi:\PP^N\to\PP^N$ be morphisms of degree at least~$2$ defined
over~$\Qbar$. Then
\[
  \dhat(\f,\psi) - h(\psi) \ll h(\f) 
  \ll \dhat(\f,\psi) + h(\psi),
\]
where the implied constants depend only on~$N$ and the degrees of~$\f$
and~$\psi$. For the upper bound, letting~$d=\deg(\f)$,
we obtain an explicit estimate of the form
\[
  h(\f) \le (d+1)\binom{N+d}{N}\dhat(\f,\psi) + O_{N,d}\bigl(h(\psi)\bigr).
\]
\end{theorem}
\begin{proof}
For convenience, let~$\l$ be the squaring map, so~$\hhat_\l$ is
the usual Weil height~$h$. Also let~$K=\binom{N+d}{N}$ as usual. 
We estimate
\begin{align*}
  h(\f)
  &\le Kd\Dhat_\l(\f)+O(1)
    &&\text{from Theorem~\ref{thm:h:c},}\\
  &\le K(d+1)\dhat(\f,\l)+O(1)
    &&\text{from Prop.~\ref{proposition:DdD},}\\
  &\le K(d+1)\bigl(\dhat(\f,\psi)+\dhat(\psi,\l)\bigr)+O(1)
    &&\text{from Lemma~\ref{lemma:triangular:inequality},} \\
  &\le K(d+1)\dhat(\f,\psi)+ 2K(d+1)\Dhat_\l(\psi)+O(1)
    &&\text{from Prop.~\ref{proposition:DdD},}\\
  &\le K(d+1)\dhat(\f,\psi)+ O\bigl(h(\psi)\bigr)
    &&\text{from Theorem~\ref{thm:h:c}.}
\end{align*}
This gives the upper bound. The lower bound is proven similarly,
we leave the details to the reader.
\end{proof}

\begin{definition}
Let~$\f:\PP^N\to\PP^N$ be a morphism defined over~$\Qbar$. The
\emph{field of definition of~$\f$}, denoted~$\QQ(\f)$, is the fixed
field of
\[
  \{\s\in\Gal(\Qbar/\QQ):\f^\s=\f\}.
\]
Equivalently, $\QQ(\f)$~is the field generated by the coordinates of
the point in $\Rat_d^N(\Qbar)=\PP^L(\Qbar)$ associated to~$\f$.
\end{definition}

\begin{corollary}
\label{thm:finiteness}
Fix a morphism $\psi:\PP^N\to\PP^N$ of degree at least~$2$ defined
over~$\Qbar$ and an integer~$d\ge2$.  Then for all~$B>0$ the set of
\emph{morphisms}~$\f\in\Rat_d^N(\Qbar)$ satisfying
\[
  \dhat(\f,\psi)\le B
\]
is a set of bounded height in~$\Rat_d^N(\Qbar)\cong\PP^L(\Qbar)$.
\par
In particular, with~$\psi$ fixed as above and for any
constants~$B,C,D$, there are only finitely many morphisms
$\f:\PP^N\to\PP^N$ defined over~$\Qbar$ and satisfying
\[
  2\le\deg(\f)\le D,\qquad
  \bigl[\QQ(\f):\QQ\bigr]\le C,\qquad\text{and}\qquad
  \dhat(\f,\psi) \le B.
\]
\end{corollary}
\begin{proof}
Using Theorem~\ref{thm:comphd}, the assumption that~$\dhat(\f,\psi)\le
B$ implies that~$h(\f)\ll B+h(\psi)$ is bounded, which proves the
first assertion. Then the second statement follows immediately from
Northcott's theorem, which says that there are only finitely many
points of bounded height and degree in projective space
(see~\cite[B.2.3]{MR1745599} or~\cite[Chapter~3,
Theorem~2.6]{MR715605}).
\end{proof}

\section{Dynamics and a $\PGL$-invariant arithmetic distance}
The dynamical properties of a morphism~$\f:\PP^N\to\PP^N$
are essentially unchanged if~$\f$ is replaced by a $\PGL$-conjugate
\[
  \f^f(P) = (f^{-1}\circ\f\circ f)(P)
  \quad\text{for some $f\in\Aut(\PP^N)=\PGL_{N+1}$.}
\]
This naturally leads one to study the quotient space
\[
  M_d^N = \bigl\{\f\in\Rat_d^N:\text{$\f$ is a morphism}\bigr\}\bigm/\PGL_{N+1}.
\]
In particular, Milnor constructed~$M_d^1(\CC)$ as a complex
orbifold~\cite{milnor:quadraticmaps} and the second author used
geometric invariant theory to construct~$M_d^1$ as a variety
over~$\QQ$ (and  as a scheme over~$\ZZ$),
see~\cite{silverman:modulirationalmaps}. We expect more generally
that~$M_d^N$ has the structure of a variety over~$\QQ$ (and a scheme
over~$\ZZ$), although this result does not seem to have yet appeared in the
literature.
\par
In any case, it is natural to define arithmetic distances and
arithmetic complexity for $\PGL$-equivalence classes of morphisms.
For convenience we write~$[\f]\in M_d^N$ for the~$\PGL$-equivalence
class containing the morphism~$\f$.

\begin{definition}
Let~$\f,\psi:\PP^N\to\PP^N$ be morphisms of degree at least~$2$
defined over~$\Qbar$. The (\emph{dynamical}) \emph{arithmetic distance
from~$[\f]$ to~$[\psi]$} is
\[
  \dhat\bigl([\f],[\psi]\bigr)
  = \inf_{f,g\in\PGL_{N+1}(\Qbar)} \dhat(\f^f,\psi^g).
\]
\end{definition}

We note some elementary properties of canonical heights and arithmetic
distances under $\PGL$-conjugation.

\begin{proposition}
Let~$\f,\psi:\PP^N\to\PP^N$ be morphisms of degree at least~$2$
defined over~$\Qbar$ and let~$f,g\in\PGL_{N+1}(\Qbar)$.
\begin{parts}
\Part{(a)}
$\hhat_{\f^f}(P) = \hhat_\f(f(P))$.
\Part{(b)}
$\dhat(\f^f,\psi^g) = \dhat(\f^{fg^{-1}},\psi)$.
\Part{(c)}
$\displaystyle
\dhat\bigl([\f],[\psi]\bigr) = 
\operatornamewithlimits{inf\vphantom{p}}_{f\in\PGL_{N+1}(\Qbar)} 
\sup_{P\in\PP^N(\Qbar)} 
\bigl|\hhat_\f(f(P))-\hhat_\psi(P)\bigr|$.
\end{parts}
\end{proposition}
\begin{proof}
{\allowdisplaybreaks
(a) Let $d=\deg(\f)$. Then
\begin{align*}
  \hhat_{\f^f}(P)
    &= \lim_{n\to\infty} \frac{1}{d^n}h\bigl((\f^f)^n(P)\bigr) \\
    &= \lim_{n\to\infty} \frac{1}{d^n}h\bigl((f^{-1}\circ\f^n\circ f)(P)\bigr)\\
    &= \lim_{n\to\infty} \frac{1}{d^n} \left(
          h\bigl((\f^n\circ f)(P)\bigr)+ O_f(1)\right)\\
    &=\hhat_\f\bigl(f(P)\bigr).
\end{align*}
\par\noindent(b)
We compute
\begin{align*}
  \dhat(\f^f,\psi^g) 
    &= \sup_{P\in\PP^N(\Qbar)} \left|\hhat_{\f^f}(P)-\hhat_{\psi^g}(P)\right|
      &&\text{definition of $\dhat$,} \\
    &= \sup_{P\in\PP^N(\Qbar)} 
        \left|\hhat_{\f^f}(g^{-1}(P))-\hhat_{\psi^g}(g^{-1}(P))\right| \\
    &= \sup_{P\in\PP^N(\Qbar)} 
        \left|\hhat_{\f^{fg^{-1}}}(P)-\hhat_{\psi}(P))\right| 
      &&\text{from (a),}\\
    &=\dhat(\f^{fg^{-1}},\psi).
\end{align*}
\par\noindent(c)
We compute
\begin{align}
  \dhat\bigl([\f],[\psi]\bigr) 
  &= \inf_{f,g\in\PGL_{N+1}(\Qbar)} \dhat(\f^f,\psi^g)
      &&\text{definition of $\dhat$,} \notag\\
  &= \inf_{f,g\in\PGL_{N+1}(\Qbar)} \dhat(\f^{fg^{-1}},\psi)
      &&\text{from (b),} \notag\\
  &= \inf_{f\in\PGL_{N+1}(\Qbar)} \dhat(\f^{f},\psi) \notag\\
  &= \operatornamewithlimits{inf\vphantom{p}}_{f\in\PGL_{N+1}(\Qbar)} 
       \sup_{P\in\PP^N(\Qbar)} 
       \bigl|\hhat_{\f^f}(P)-\hhat_\psi(P)\bigr|
      &&\text{definition of $\dhat$,} \notag\\
  &= \operatornamewithlimits{inf\vphantom{p}}_{f\in\PGL_{N+1}(\Qbar)} 
       \sup_{P\in\PP^N(\Qbar)} 
       \bigl|\hhat_\f(f(P))-\hhat_\psi(P)\bigr|
      &&\text{from (a).}  
  \tag*{\qedsymbol}
\end{align}
}%
\renewcommand{\qedsymbol}{}
\end{proof}

We conclude by asking if there is a single~$f\in\PGL_{N+1}(\Qbar)$
that achieves the infimum in the definition of arithmetic distance
on~$M_d^N$.

\begin{question}
Let~$\f,\psi:\PP^N\to\PP^N$ be morphisms of degree at least~$2$
defined over~$\Qbar$.
Does there always exist an~$f\in\PGL_{N+1}(\Qbar)$ such that
\[
  \dhat\bigl([\f],[\psi]\bigr) = \dhat(\f^f,\psi)?
\]
\end{question}



\begin{thebibliography}{10}


\bibitem{MR1627110}
Ahmed Abbes.
\newblock Hauteurs et discr\'etude (d'apr\`es {L}. {S}zpiro, {E}. {U}llmo et
  {S}. {Z}hang).
\newblock {\em Ast\'erisque}, (245):Exp.\ No.\ 825, 4, 141--166, 1997.
\newblock S\'eminaire Bourbaki, Vol.\ 1996/97.

\bibitem{MR1260106}
J.-B. Bost, H.~Gillet, and C.~Soul{\'e}.
\newblock Heights of projective varieties and positive {G}reen forms.
\newblock {\em J. Amer. Math. Soc.}, 7(4):903--1027, 1994.


\bibitem{MR1255693}
Gregory~S. Call and Joseph~H. Silverman.
\newblock Canonical heights on varieties with morphisms.
\newblock {\em Compositio Math.}, 89(2):163--205, 1993.

\bibitem{MR1745599}
Marc Hindry and Joseph~H. Silverman.
\newblock {\em Diophantine geometry}, volume 201 of {\em Graduate Texts in
  Mathematics}.
\newblock Springer-Verlag, New York, 2000.
\newblock An introduction.

\bibitem{arxiv0510633}
Shu Kawaguchi.
\newblock Canonical heights for random iterations in certain varieties, 2005.
\newblock \url{math.AG/0510633}, 27 pages, submitted for publication.

\bibitem{KSEqualCanHt}
Shu Kawaguchi and Joseph~H. Silverman.
\newblock Dynamics of projective morphisms having identical canonical heights,
  February 2006.
\newblock submitted for publication.

\bibitem{MR715605}
Serge Lang.
\newblock {\em Fundamentals of {D}iophantine geometry}.
\newblock Springer-Verlag, New York, 1983.

\bibitem{MR0173676}
Ju.~I. Manin.
\newblock The {T}ate height of points on an {A}belian variety, its variants and
  applications.
\newblock {\em Izv. Akad. Nauk SSSR Ser. Mat.}, 28:1363--1390, 1964.

\bibitem{milnor:quadraticmaps}
John Milnor.
\newblock Geometry and dynamics of quadratic rational maps.
\newblock {\em Experiment. Math.}, 2(1):37--83, 1993.
\newblock With an appendix by the author and Lei Tan.

\bibitem{MR0179173}
A.~N{\'e}ron.
\newblock Quasi-fonctions et hauteurs sur les vari\'et\'es ab\'eliennes.
\newblock {\em Ann. of Math. (2)}, 82:249--331, 1965.

\bibitem{MR942161}
Joseph~H. Silverman.
\newblock Computing heights on elliptic curves.
\newblock {\em Math. Comp.}, 51(183):339--358, 1988.

\bibitem{silverman:modulirationalmaps}
Joseph~H. Silverman.
\newblock The space of rational maps on {$\mathbf{P}\sp 1$}.
\newblock {\em Duke Math. J.}, 94(1):41--77, 1998.

\bibitem{MR866121}
Heinz~M. Tsch{\"o}pe and Horst~G. Zimmer.
\newblock Computation of the {N}\'eron-{T}ate height on elliptic curves.
\newblock {\em Math. Comp.}, 48(177):351--370, 1987.

\bibitem{MR1254133}
Shouwu Zhang.
\newblock Positive line bundles on arithmetic varieties.
\newblock {\em J. Amer. Math. Soc.}, 8(1):187--221, 1995.

\end{thebibliography}

\end{document}